\documentclass[a4paper,final]{amsart}%arXiv
\usepackage{amsmath, amssymb, amsthm}
\usepackage{enumitem,fullpage}
\usepackage{showkeys}
\usepackage{color}
\usepackage{stackrel}
\usepackage{here}
\usepackage{tikz}
\usetikzlibrary{intersections, calc, arrows, cd}
\numberwithin{equation}{section}
\numberwithin{figure}{section}
\allowdisplaybreaks

\theoremstyle{definition}
\newtheorem{thm}{Theorem}[section]
\newtheorem{prp}[thm]{Proposition}
\newtheorem{lem}[thm]{Lemma}
\newtheorem{cor}[thm]{Corollary}
\newtheorem{dfn}[thm]{Definition}

\newtheorem{rmk}[thm]{Remark}
\newtheorem{ex}[thm]{Example}

\newtheorem{fct}[thm]{Fact}
\newtheorem*{thm*}{Theorem}
\newtheorem*{prp*}{Propostion}
\newtheorem*{lem*}{Lemma}
\newtheorem*{dfn*}{Definition}
\newtheorem*{nt*}{Notattion}
\newtheorem*{fct*}{Fact}
\newtheorem*{rmk*}{Remark}
\newtheorem*{ex*}{Example}

\newcommand{\xr}[1]{\xrightarrow{\, #1 \, }}

\newcommand{\xba}[2]{\stackbin[#2]{#1}{\rightleftarrows}}
\newcommand{\inj}{\hookrightarrow}

\newcommand{\simto}{\xr{\sim}}

\newcommand{\bbF}{\mathbb{F}}

\newcommand{\bbZ}{\mathbb{Z}}

\newcommand{\bbk}{\Bbbk}

\DeclareMathOperator{\catmod}{mod}

\newcommand{\catA}{\mathcal{A}}

\newcommand{\catC}{\mathcal{C}}

\newcommand{\catT}{\mathcal{T}}

\DeclareMathOperator{\id}{id}
\DeclareMathOperator{\Id}{Id}

\DeclareMathOperator{\Ob}{Ob}

\DeclareMathOperator{\Mor}{Mor}

\DeclareMathOperator{\Hom}{Hom}
\DeclareMathOperator{\End}{End}
\DeclareMathOperator{\Ext}{Ext}
\DeclareMathOperator{\Ima}{Im}
\DeclareMathOperator{\Ker}{Ker}

\DeclareMathOperator{\Proj}{Proj}

\newcommand{\LL}{\Lambda}

\newcommand{\Sg}{\Sigma}

\newcommand{\modu}{\,\mathrm{mod}\,}

\newenvironment{enur}{\begin{enumerate}[nosep, label=(\roman*)]}{\end{enumerate}}
\newenvironment{enua}{\begin{enumerate}[nosep, label=(\arabic*)]}{\end{enumerate}}

\makeatletter
\newcommand{\colim@}[2]{%
	\vtop{\m@th\ialign{##\cr
			\hfil$#1\operator@font colim$\hfil\cr
			\noalign{\nointerlineskip\kern1.5\ex@}#2\cr
			\noalign{\nointerlineskip\kern-\ex@}\cr}}%
}
\newcommand{\colim}{%
	\mathop{\mathpalette\colim@{\rightarrowfill@\textstyle}}\nmlimits@
}
\makeatother
\begin{document}

\title{A note on Grothendieck groups of periodic derived categories}
\author{Shunya Saito}
\date{December 20, 2021}
\address{Graduate School of Mathematics, Nagoya University, Chikusa-ku, Nagoya. 464-8602, Japan}
\email{m19018i@math.nagoya-u.ac.jp}

\begin{abstract}
We determine Grothendieck groups of periodic derived categories.
In particular,
we prove that the Grothendieck group of
the $m$-periodic derived category of finitely generated modules over an Artin algebra
is a free $\bbZ$-module if $m$ is even
but an $\bbF_2$-vector space if $m$ is odd.
Its rank is equal to the number of isomorphism classes of simple modules in both cases.
As an application,
we prove that the number of non-isomorphic summands of 
a strict periodic tilting object $T$,
which was introduced in \cite{S21} as a periodic analogue of tilting objects,
is independent of the choice of $T$.
\end{abstract}

\maketitle
\tableofcontents

%%%%%%%%%%%%%%%%%%%%%%%%%%%%%%%%%%%%%%%%%%%%%%%%%%%%%%%%%%%%%%%%%%%%%%%%%%%%%%%%%%%%%%%%%
%%%%%%%%%%%%%%%%%%%%%%%%%%%%%%%%%%%%%%%%%%%%%%%%%%%%%%%%%%%%%%%%%%%%%%%%%%%%%%%%%%%%%%%%%
\section{Introduction}\label{s:Intro}
%%%%%%%%%%%%%%%%%%%%%%%%%%%%%%%%%%%%%%%%%%%%%%%%%%%%%%%%%%%%%%%%%%%%%%%%%%%%%%%%%%%%%%%%%
\subsection{Background}
The representation theory of Artin algebras studies
properties of Artin algebras which is preserved under \emph{Morita equivalence}.
Two Artin algebras are Morita equivalent
if the categories $\catmod \LL$ of finitely generated modules over them are equivalent.
\emph{Grothendieck groups} control many invariants of Morita equivalence.
They are defined for each abelian category and invariant under equivalence of them,
and thus Grothendieck groups themselves are a Morita invariant.
For example,
the Grothendieck group $K_0(\catmod \LL)$ of finitely generated modules
over an Artin algebra $\LL$ is a free $\bbZ$-module whose rank is
equal to the number of isomorphism classes of simple $\LL$-modules.
Thus the number of isomorphism classes of simple modules is a Morita invariant.

\emph{Derived equivalence} is a more flexible framework 
to study Artin algebras than Morita equivalence.
Two Artin algebras are derived equivalent
if the bounded derived categories $D^b(\catmod \LL)$ 
of finitely generated modules over them are triaungulated equivalent.
Of course, Morita equivalence implies derived equivalence.
We can define Grothendieck groups of triangulated categories
and have an isomorphism $K_0(\catmod \LL) \simeq K_0(D^b(\catmod \LL))$.
It means that
Grothendieck groups are not only a Morita invariant but also a derived invariant.
In particular,
the number of isomorphism classes of simple modules is still a derived invariant.
The purpose of this paper is 
to give an analogy of the above results for \emph{periodic derived categories}
(Theorem \ref{thm:main} and Corollary \ref{cor:main})
and to apply it to \emph{periodic tilting theory} (Corollary \ref{cor:const}).

Let $m>0$ be a positive integer.
\emph{$m$-periodic complexes} are a $\bbZ/m\bbZ$-graded version of usual complexes.
For example, a $2$-periodic complex $V$ consists of 
objects and morphisms $V^0 \xba{d^0}{d^1} V^1$ with $d^1d^0=d^0d^1=0$.
For an abelian category $\catA$,
the \emph{$m$-periodic derived category} $D_m(\catA)$ of $\catA$ is 
the localization of the category of $m$-periodic complexes
with respect to quasi-isomorphisms.
See \S \ref{ss:per der} for the precise definition.
These were introduced by Peng and Xiao in \cite{PX97} 
to construct a categorification of full semisimple Lie algebras via Ringel-Hall algebras.
Inspired by this work,
Bridgeland used $2$-periodic derived categories of hereditary algebras
to construct full quantum groups of symmetric Kac-Moody Lie algebras in \cite{Br13}.
Motivated by these studies,
several authors analyzed the structure of $m$-periodic derived categories.
See \cite{Fu12,Go13,Zhao14,St18}.

%%%%%%%%%%%%%%%%%%%%%%%%%%%%%%%%%%%%%%%%%%%%%%%%%%%%%%%%%%%%%%%%%%%%%%%%%%%%%%%%%%%%%%%%%
\subsection{Main results}
Let $m>0$ be a positive integer.
The following theorem is the main result of this paper,
which will be shown in \S \ref{s:gro per der}.
It indicates $m$-periodic derived categories behave like usual derived categories 
if $m$ is even.
In contrast, they behave strangely if $m$ is odd.
\begin{thm}\label{thm:main}
For an essentially small enough projective abelian category $\catA$
of finite global dimension, we have an isomorphism
\[
K_0(D_m(\catA))\simeq
\begin{cases}
K_0(\catA) & \text{if $m$ is even},\\
K_0(\catA)_{\bbF_2}:=K_0(\catA)\otimes_{\bbZ}\bbF_2 & \text{if $m$ is odd},
\end{cases}
\]
induced by the natural functor $\catA \to D_m(\catA)$,
where $\bbF_2$ is the finite field of two elements.
\end{thm}
The result of the even periodic case
is a direct generalization of \cite[Proposition 2.11]{Fu12},
but the odd periodic case is a new one and 
includes a new insight of odd periodic triangulated categories.

As an immediate corollary, we have the following.
\begin{cor}\label{cor:main}
Let $\LL$ be an Artin algebra of finite global dimension,
and let $n$ be the number of isomorphism classes of simple modules. 
Then we have
\[
K_0(D_m(\catmod\LL))\simeq
\begin{cases}
\bbZ^{n} & \text{if $m$ is even},\\
\bbF_2^{n} & \text{if $m$ is odd}.
\end{cases}
\]
\end{cor}

We explain a motivation of Theorem \ref{thm:main} and Corollary \ref{cor:main}.
In \cite{S21}, the author proves the periodic tilting theorem,
which gives a sufficient condition for a given triangulated category to be equivalent to
the periodic derived category of an algebra.
We review some definitions and results in \cite{S21}.

For a positive integer $m>0$,
a triangulated category $\catT$ is called \emph{$m$-periodic} 
if its suspension functor $\Sg :\catT \simto \catT$ satisfies
$\Sg^m \simeq \Id_{\catT}$ as additive functors.
\begin{dfn}
Let $\catT$ be an $m$-periodic triangulated category.
\begin{enua}
\item
An object $T\in\catT$ is \emph{$m$-periodic tilting} 
if it satisfies $\Hom_{\catT}(T, \Sg^i T)=0$ for any $i\not\in m\bbZ$ 
and the smallest thick triangulated category containing $T$ coincides with $\catT$.
\item
An $m$-periodic tilting object $T\in \catT$ is called \emph{strict}
if the global dimension of the endomorphism algebra $\End_{\catT}(T)$ is less than $m$.
\end{enua}
\end{dfn}

We will not inform the \emph{algebraic} and \emph{idempotent complete} conditions
in the following theorem.
See \cite{S21} for the detail.
These mild assumptions are satisfied by 
almost all concrete triangulated categories appearing in the study 
of representations of algebras.
\begin{thm}[The periodic tilting theorem {\cite[Corollary 5.4]{S21}}]\label{thm:m-tilt}
Let $\catT$ be an idempotent complete algebraic $m$-periodic 
triangulated category over a perfect field $\bbk$ 
(e.g., if $\bbk$ is an algebraically closed field).
Suppose that $\Hom_\catT(X,Y)$ is finite dimensional over $\bbk$ 
for all objects $X,Y\in\catT$.
If $\catT$ has a strict $m$-periodic tilting object $T$,
then there exists a triangulated equivalence $\catT \to D_m(\catmod \LL)$,
where $\LL :=\End_{\catT}(T)$.
\end{thm}
The periodic tilting theorem and periodic tilting objects are
periodic analogue of usual tilting theorem and tilting objects (cf. \cite{Tilt07}).
Hence we expect that periodic tilting objects have properties similar to the usual one.
However, there is the following example which is taught by Professor Osamu Iyama 
in the conference \emph{Algebraic Lie Theory and Representation Theory, 2021}.
\begin{ex}\label{ex:Iyama}
Let $\bbk$ be a perfect field, 
and $\bbk A_3$ the path algebra of
the quiver $1 \leftarrow 2 \leftarrow 3$ of type $A_3$.
The Auslander-Reiten quiver of $D_2(\catmod \bbk A_3)$ is the following. 
(See Example \ref{ex:path m})
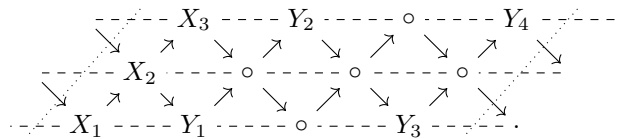
\begin{figure}[H]
\centering
\begin{tikzpicture}
\node (1) at (0,0) {$X_1$};
\node [above right of=1] (12) {$X_2$};
\node [above right of=12] (123) {$X_3$};
\node [below right of=12] (2) {$Y_1$};
\node [above right of=2] (23) {$\circ$};
\node [above right of=23] (234) {$Y_2$};
\node [below right of=23] (3) {$\circ$};
\node [above right of=3] (34) {$\circ$};
\node [above right of=34] (341) {$\circ$};
\node [below right of=34] (4) {$Y_3$};
\node [above right of=4] (41) {$\circ$};
\node [above right of=41] (412) {$Y_4$};
 \draw[->] (1) -- (12);
 \draw[->] (12) -- (123);
 \draw[->] (12) -- (2);
 \draw[->] (123) -- (23);
 \draw[->] (2) -- (23);
 \draw[->] (23) -- (234);
 \draw[->] (23) -- (3);
 \draw[->] (234) -- (34);
 \draw[->] (3) -- (34);
 \draw[->] (34) -- (341);
 \draw[->] (34) -- (4);
 \draw[->] (341) -- (41);
 \draw[->] (4) -- (41);
 \draw[->] (41) -- (412);
 \draw[dashed] (2) -- (1);
 \draw[dashed] (3) -- (2);
 \draw[dashed] (4) -- (3);
 \draw[dashed] (23) -- (12);
 \draw[dashed] (34) -- (23);
 \draw[dashed] (41) -- (34);
 \draw[dashed] (234) -- (123);
 \draw[dashed] (341) -- (234);
 \draw[dashed] (412) -- (341);
\node [above left of=1] ('1) {};
\node [above left of=12] ('12) {};
\node [above left of=123] ('123) {};
\node [below right of=41] (41') {.};
\node [below right of=412] (412') {};
\node [above right of=412'] (4123') {};
 \draw[->] ('1) -- (1);
 \draw[->] ('12) -- (12);
 \draw[->] (41) -- (41');
 \draw[->] (412) -- (412');
 \draw[dotted] (-0.75,-0.1) -- (0.75,1.6);
 \draw[dotted] (5,-0.1) -- (6.5,1.6);
 \draw[dashed] (1) -- (-1,0);
 \draw[dashed] (12) -- ('1);
 \draw[dashed] (123) -- ('12);
 \draw[dashed] (41') -- (4);
 \draw[dashed] (412') -- (41);
 \draw[dashed] (4123') -- (412);
\end{tikzpicture}
\end{figure}
Then $X:=\bigoplus_{i=1}^3 X_i$ and $Y:=\bigoplus_{i=1}^4 Y_i$
are both $2$-periodic tilting objects in $D_2(\catmod \bbk A_3)$.
Thus
the number of non-isomorphic summands of a periodic tilting object is not constant,
while the number for the usual one is constant.
\end{ex}
In this example, we observe that 
$\End(X)\simeq \bbk A_3$ and 
$\End(Y)$ is isomorphic to a self-injective Nakayama algebra,
and hence $X$ is strict but $Y$ is not.
We expect that the number of 
non-isomorphic summands of a \emph{strict} periodic tilting object is constant.
This is true by Corollary \ref{cor:main} and Theorem \ref{thm:m-tilt}.

\begin{cor}\label{cor:const}
Fix a positive integer $m>0$.
Let $\catT$ be an idempotent complete algebraic $m$-periodic 
triangulated category over a perfect field $\bbk$.
Suppose that $\Hom_\catT(X,Y)$ is finite dimensional over $\bbk$ 
for all objects $X,Y\in\catT$.
Then the number of non-isomorphic summands of a strict periodic tilting object 
is constant.
\end{cor}
\begin{proof}
Suppose $T_i\in \catT \,(i=1,2 )$ are strict $m$-periodic tilting objects 
and set $\LL_i :=\End_{\catT}(T_i)$.
Then we have two triangulated equivalences
$\catT \simto D_m(\catmod\LL_i)$ by Theorem \ref{thm:m-tilt}.
A triangulated equivalence $D_m(\catmod \LL_1) \simeq \catT \simeq D_m(\catmod \LL_2)$
induces an isomorphism $K_0(D_m(\catmod \LL_1))\simeq K_0(D_m(\catmod \LL_2))$
on the Grothendieck groups.
Hence $\LL_1$ and $\LL_2$ have the same number of isomorphism classes of simple modules
by Corollary \ref{cor:main}.
Because the number of non-isomorphic summands of $T_i$ is equal to
the number of isomorphism classes of simple modules over $\LL_i$,
the corollary follows.
\end{proof}

%%%%%%%%%%%%%%%%%%%%%%%%%%%%%%%%%%%%%%%%%%%%%%%%%%%%%%%%%%%%%%%%%%%%%%%%%%%%%%%%%%%%%%%%%
\subsection*{Organization}
This paper is organized as follows.
In Section \ref{s:pre},
we collect basic properties of triangulated categories
and periodic derived categories
which we use throughout this paper.
In Section \ref{s:gro per tri},
we investigate general properties of 
the Grothendieck group of a periodic triangulated category.
In particular, we deal with the relationship between
cohomological functors on periodic triangulated categories
and homomorphisms between the Grothendieck groups of them.
In Section \ref{s:gro per der},
we give a proof of Theorem \ref{thm:main}.

%%%%%%%%%%%%%%%%%%%%%%%%%%%%%%%%%%%%%%%%%%%%%%%%%%%%%%%%%%%%%%%%%%%%%%%%%%%%%%%%%%%%%%%%%
\subsection*{Acknowledgement}
The author would like to thank Osamu Iyama 
for suggesting the problem and giving an interesting example (Example \ref{ex:Iyama}).
This work is supported by JSPS KAKENHI Grant Number JP21J21767.

%%%%%%%%%%%%%%%%%%%%%%%%%%%%%%%%%%%%%%%%%%%%%%%%%%%%%%%%%%%%%%%%%%%%%%%%%%%%%%%%%%%%%%%%%
%%%%%%%%%%%%%%%%%%%%%%%%%%%%%%%%%%%%%%%%%%%%%%%%%%%%%%%%%%%%%%%%%%%%%%%%%%%%%%%%%%%%%%%%%
\section{Preliminaries}\label{s:pre}
%%%%%%%%%%%%%%%%%%%%%%%%%%%%%%%%%%%%%%%%%%%%%%%%%%%%%%%%%%%%%%%%%%%%%%%%%%%%%%%%%%%%%%%%%
\subsection{Triangulated categories}\label{ss:tri}
In this subsection,
we gather basic notions on triangulated categories,
which we will use.
Throughout this paper,
we assume that all categories and functors are additive.
We denote by $\Sg$ the suspension functor of a triangulated category.

Let $\catT$ be an essentially small triangulated category,
that is, the isomorphism classes $[\catT]$ of objects 
and $\Hom_{\catT}(X,Y) \, (X,Y\in\catT)$ form sets.
For an object $X\in \catT$, the isomorphism class of $X$ 
is denoted by $[X]\in[\catT]$.
The Grothendieck group $K_0(\catT)$ of $\catT$ is defined as
a quotient of a free abelian group :
\[
\bigoplus_{[X]\in [T]}\bbZ [X] \bigm/ \left< [X]-[Y]+[Z] \mid \text{$X\to Y\to Z\to \Sg X$ is an exact triangle } \right>.
\]
The residue class of $[X] \in [\catT]$ in $K_0(\catT)$
is also denoted by $[X]$.
It is obvious that a triangle equivalence $\catT \simto \catT'$ induces 
an isomorphism $K_0(\catT) \simto K_0(\catT')$.

For an essentially small abelian category $\catA$,
we can similarly define the Grothendieck group $K_0(\catA)$ of $\catA$.
Triangulated categories often relate to abelian categories
and it gives a homomorphism between Grothendieck groups of them.
\begin{dfn}\label{dfn:delta,cohom}
Let $\catA$ be an abelian category and let $\catT$ be a triangulated category.
\begin{enua}
\item
A \emph{$\delta$-functor} $\catA\to\catT$ from $\catA$ to $\catT$
is a pair $(F,\delta)$ of a functor $F:\catA\to\catT$
and functorial morphisms $\delta_{W,U}:\Ext^1_{\catA}(W,U)\to \Hom_\catT(FW,\Sg FU)$
for all $U,W\in\catA$ such that for any exact sequence 
$\epsilon : 0\to U \xr{f}V\xr{g}W\to0$ in $\catA$,
\[
FU\xr{Ff}FV\xr{Fg}FW\xr{\delta(\epsilon)}\Sg FU
\]
is an exact triangle in $\catT$.
If no ambiguity can arise,
we will often say that $F:\catA \to \catT$ is a $\delta$-functor.
\item
A functor $F : \catT \to \catA$ is said to be \emph{cohomological}
if for any exact triangle $X\to Y\to Z \to \Sg X$ in $\catT$,
$FX \to FY \to FZ$ is an exact sequence in $\catA$.
We set $F^i:=F\circ \Sg^m : \catT \to \catA$.
\end{enua}
\end{dfn}
Typical examples of $\delta$-functors and cohomological functors are 
the natural inclusion $\catA \to D(\catA)$ from an abelian category to its derived category and the $0$th cohomology functor $H^0 : D(\catA) \to \catA$, respectively.
In \S \ref{ss:per der},
we will see the counterparts of these examples
in periodic derived categories $D_m(\catA)$. 
A $\delta$-functor $\catA \to \catT$ naturally induces 
a homomorphism $K_0(\catA) \to K_0(\catT)$
but a cohomological functor $\catT \to \catA$ does not induce
a homomorphism $K_0(\catT) \to K_0(\catA)$ in general.
It is a difficulty we deal with in this paper.

\begin{dfn}\label{dfn:periodic}
Let $\catT$ be a triangulated category.
\begin{enua}
\item
For a positive integer $m>0$, $\catT$ is \emph{$m$-periodic}
if $\Sg^m \simeq \Id_{\catT}$ as additive functors.
\item
The \emph{period} of $\catT$
is the smallest positive integer $m$
such that $\catT$ is $m$-periodic.
\end{enua}
\end{dfn}
We study Grothendieck groups of periodic triangulated categories 
in \S \ref{s:gro per tri}.

%%%%%%%%%%%%%%%%%%%%%%%%%%%%%%%%%%%%%%%%%%%%%%%%%%%%%%%%%%%%%%%%%%%%%%%%%%%%%%%%%%%%%%%%%
\subsection{Periodic derived categories}\label{ss:per der}
In this subsection, we give a review of periodic derived categories.
See \cite[\S 3]{S21} for a detailed account.
Fix a positive integer $m>0$.
$\bbZ_m$ denotes the cyclic group of order $m$.
Roughly speaking, an $m$-periodic complex is a $\bbZ_m$-graded complex.
In the following definition, replacing $\bbZ_m$ to $\bbZ$, 
we get the usual notion of complexes.
\begin{dfn}\label{dfn:DG m-cpx}
Let $\catC$ be an additive category.
\begin{enua}
\item
An \emph{$m$-periodic complex} $V$ is a family $(V,d_V) = (V^i, d_V^i)_{i\in\bbZ_m}$
of objects $V_i\in \catC$ and morphisms $d_V^i:V^i \to V^{i+1}$ in $\catC$
satisfying $d_V^{i+1}d_V^i=0$ for all $i \in \bbZ_m$. 
\item
A chain map $f:V\to W$ between $m$-periodic complexes $V$ and $W$
is a family $(f_i)_{i\in\bbZ_m}$ of morphisms in $\catC$
satisfying $f^{i+1}d_V^{i}=d_W^if^i$ for all $i \in \bbZ_m$.
\item
$C_m(\catC)$ denotes the category of $m$-periodic complexes and chain maps.
\end{enua}
\end{dfn}
\begin{ex}
Let $\catC$ be an additive category.
\begin{enua}
\item
A $1$-periodic complex is a morphism $d:V \to V$ in $\catC$ with $d^2=0$.
\item
A $2$-periodic complex is a diagram $V^0 \xba{d^0}{d^1} V^1$ in $\catC$ 
with $d^1d^0=d^0d^1=0$.
\end{enua}
\end{ex}

Two chain maps $f,g: V \to W$ of $m$-periodic complexes is \emph{homotopic}
if there exist $s^i:V^i \to W^{i-1}\,(i\in \bbZ_m)$ 
with $f^i-g^i=d_W^{i-1}s^i+s^{i+1}d_V^i$ for all $i\in \bbZ_m$.
This gives rise to the \emph{homotopy category} $K_m(\catC)$ of $m$-periodic complexes.
The shift functor $[1]: C_m(\catC) \to C_m(\catC)$ is defined by
\[
V \mapsto V[1]:=(V^{i+1},-d_V^{i+1})_{i\in\bbZ_m}.
\]
The homotopy category $K_m(\catC)$ 
with the shift functor $[1]:K_m(\catC) \to K_m(\catC)$ as the suspension functor
is a triangulated category.

For an abelian category $\catA$,
the category $C_m(\catA)$ is also an abelian category.
A sequence $0 \to U \xr{f} V \xr{g} W \to 0$
in $C_m(\catA)$ is exact
if and only if
$0 \to U^i \xr{f^i} V^i \xr{g^i} W^i \to 0$
is exact in $\catA$ for all $i\in \bbZ_m$.
Define the $i$th \emph{cohomology} of $V\in C_m(\catA)$ 
by $H^i(V):=\Ker d^i_V/\Ima d^{i-1}_V$ for $i\in\bbZ_m$.
It gives rise to a functor $H^i : K_m(\catA) \to \catA$ for all $i\in \bbZ_m$.
A chain map $f:V\to W$ of $m$-periodic complexes is a \emph{quasi-isomorphism}
if $H^i(f):H^i(V) \to H^i(W)$ is an isomorphism for all $i\in \bbZ_m$.

\begin{dfn}
For an abelian category $\catA$,
the \emph{$m$-periodic derived category} $D_m(\catA)$
is the localization of $K_m(\catA)$
with respect to quasi-isomorphisms.
\end{dfn}
The category $D_m(\catA)$ is a triangulated category 
and the canonical functor $K_m(\catA) \to D_m(\catA)$ is a triangulated functor.

\begin{rmk}\label{rmk:period}
The $m$-periodic derived category $D_m(\catA)$ of an abelian category $\catA$
is a fundamental example of periodic triangulated category 
(See Definition \ref{dfn:periodic}),
but its period is not necessarily $m$.
The period depends on the parity of $m$.
This phenomenon is caused by the change of signs of differential by the shift functor.
For example, the shift of a $1$-periodic complex is $(M,d)[1]=(M,-d)$.
Hence $(M,d)$ and $(M,d)[1]$ is not isomorphic in general.
There are three cases for the period $p$ of $D_m(\catA)$
(See \cite[Proposition 5.1]{S21}):
\begin{enur}
\item $m$ is even and $p=m$,
\item $m$ is odd and $p=m$, and
\item $m$ is odd and $p=2m$.
\end{enur}
\end{rmk}

The $i$th cohomology functor $H^i: K_m(\catA) \to \catA$ induces
a functor $D_m(\catA)\to \catA$.
We also denote it by $H^i$.
It is an advantage of localizing the category $K_m(\catA)$
that an exact sequence in $\catA$ gives an exact triangle in $D_m(\catA)$.

\begin{fct}[{\cite[Propostion 3.12, 3.19]{S21}}]\label{fct:delta}
Let $\catA$ be an abelian category.
\begin{enua}
\item
The natural functor $C_m(\catA) \to D_m(\catA)$ is a $\delta$-functor
(See Definition \ref{dfn:delta,cohom} (1)).
\item
The $i$th cohomology functor $H^i : D_m(\catA)\to \catA$ is
a cohomological functor (See Definition \ref{dfn:delta,cohom} (2)).
\end{enua}
\end{fct}

The purpose of this paper is
to study the Grothendieck groups of periodic derived categories.
When do we define the Grothendieck groups of periodic derived categories?
In other words when are periodic derived categories essentially small?
The following theorem answers this question.
See Proposition \ref{prp:summ} below.
\begin{fct}[{\cite[Lemma 9.5]{Go13}}, cf. {\cite[Corollary 3.28]{S21}}]\label{fct:K-proj}
Let $\catA$ be an enough projective abelian category of finite global dimension,
and $\Proj \catA$ is the full subcategory of projective objects in $\catA$.
Then the natural functor $K_m(\Proj \catA) \to D_m(\catA)$ is a triangulated equivalence.
\end{fct}

The following fact implies surjectivity of
a homomorphism $K_0(\catA) \to K_0(D_m(\catA))$
induced by the natural inclusion $\catA \to D_m(\catA)$.
See Proposition \ref{prp:summ} below.
\begin{fct}[{\cite[Proposition 9.7]{Go13}, cf. \cite[Lemma 3.26]{S21}}]\label{fct:stalk}
Let $\catA$ be an enough projective abelian category of finite global dimension.
Then the smallest triangulated subcategory of $D_m(\catA)$ containing $\catA$
coincides with $D_m(\catA)$.
\end{fct}

We summarize the facts about periodic derived categories
and rephrase them as statements about its Grothendieck groups.
\begin{prp}\label{prp:summ}
Let $\catA$ be an essentially small enough projective abelian category 
of finite global dimension.
\begin{enua}
\item
$D_m(\catA)$ is essentially small.
In particular, we can define the Grothendieck group of $D_m(\catA)$.
\item
The natural functor $\catA \to D_m(\catA)$ is a $\delta$-functor.
In particular, we have a induced homomorphism 
$\psi : K_0(\catA) \to K_0(D_m(\catA))$.
\item
The smallest triangulated subcategory of $D_m(\catA)$ containing $\catA$
coincides with $D_m(\catA)$.
In particular,
the homomorphism $\psi : K_0(\catA) \to K_0(D_m(\catA))$ is surjective.
\end{enua}
\end{prp}
\begin{proof}
(1)
It is not obvious that $\Hom_{D_m(\catA)}(V,W)$ forms a set in general
since $D_m(\catA)$ is the localization of the category $K_m(\catA)$.
However, by Fact \ref{fct:K-proj}, 
the natural functor $K_m(\Proj\catA) \simto D_m(\catA)$ is an equivalence.
Note that $\Mor_{\catA}:=\bigcup_{M,N\in [\catA]}\Hom_{\catA}(M,N)$
forms a set because $\catA$ is essentially small.
The category $K_m(\Proj\catA)$ is clearly essentially small 
since $\Hom_{K_m(\catA)}(V,W)$ is a quotient of a set 
$\Hom_{C_m(\catA)}(V,W) \subset \prod_{i=0}^{m-1}\Hom_{\catA}(V^i,W^i)$
and $[K_m(\Proj\catA)] \subset \prod_{i=0}^{m-1}\Mor_{\catA}$.
Thus $D_m(\catA)$ is also essentially small.

(2)
Since the natural inclusion $\catA \inj C_m(\catA)$ is exact
and the natural functor $C_m(\catA) \to D_m(\catA)$ is a $\delta$-functor,
their composition $\catA \to D_m(\catA)$ is also a $\delta$-functor.
Thus an exact sequence $0\to L \to M\to N\to 0$ in $\catA$
gives an exact triangle $L\to M\to N \to L[1]$ in $D_m(\catA)$.
It implies
$\psi : K_0(\catA) \ni [M] \mapsto [M] \in K_0(D_m(\catA))$
is a well-defined homomorphism. 

(3)
For a class $S$ of objects in a triangulated category,
it is well-known that
an object of the smallest triangulated category containing $S$
is a (finite) iterated extension of shifts of objects of $S$.
Hence Fact \ref{fct:stalk} implies an object of $D_m(\catA)$ is
an iterated extension of shifts of objects of $\catA$,
and thus $\psi$ is surjective.
\end{proof}

Finally,
we explain the relationship between periodic complexes and usual complexes.
We do not use the following results and explanations in this paper
but it gives a good picture of periodic derived categories.
Let $\catA$ be an abelian category.
$C(\catA)$ (resp. $C^b(\catA)$) denotes
the category of usual (resp. bounded) complexes over $\catA$.
Define functors
\[
\iota : C_m(\catA) \to C(\catA),\quad
V \mapsto \left(V^{(i \modu m)}, d_V^{(i \modu m)}  \right)_{i\in\bbZ}
\]
and
\[
\pi : C^b(\catA) \to C_m(\catA),\quad
V \mapsto \left( \bigoplus_{j\equiv i \modu m} V^{j}, \bigoplus_{j\equiv i \modu m} d_V^{j}\right)_{i\in \bbZ_m}
\]
The functors $\iota$ and $\pi$ preserve quasi-isomorphisms,
and induce triangulated functors $\iota : D_m(\catA) \to D(\catA)$ 
and $\pi : D^b(\catA) \to D_m(\catA)$, respectively.
The functor $\pi : D^b(\catA) \to D_m(\catA)$ is called the \emph{covering functor}.
This name comes from the following fact.
\begin{fct}[{\cite[Corollary 3.29]{S21}}]\label{fct:cov}
Let $\catA$ be an enough projective abelian category 
of finite global dimension.
For any $V, W \in D^b(\catA)$,
we have $\Hom_{D_m(\catA)}(\pi V, \pi W)=\bigoplus_{i\in\bbZ}\Hom_{D^b(\catA)}(V,W[mi])$.
\end{fct}
For an additive category $\catC$ and an auto equivalence $F:\catC \simto \catC$,
the orbit category $\catC/ F$ of $\catC$ by $F$ is defined by
\[
\Ob(\catC/ F):=\Ob \catC,\quad
\Hom_{\catC/F}(X,Y):=\bigoplus_{i\in\bbZ} \Hom_{\catC}(X,F^iY).
\]
The composition of two morphisms $f:X \to F^p Y$ and $g:Y\to F^q Z$
is defined by $(F^pg)f : X \to F^{p+q}Z$.
The natural functor $\pi : \catC \to \catC/F$ is called 
the \emph{covering functor} in general.
The identity $\id_{FX}$ gives rise to 
a natural isomorphism $\pi FX \to \pi X$ for all $X\in \catC$.
Roughly speaking, the orbit category $\catC/F$ is obtained by
identifying $\{F^n\}_{n\in \bbZ}$-orbits of objects of $\catC$.

Fact \ref{fct:cov} means that $\Ima \pi \subset D_m(\catA)$ 
is the orbit category $D^b(\catA)/[m]$ of the bounded derived category 
by the $m$-shift functor.
In fact, $D_m(\catA)$ is 
the smallest triangulated category containing the orbit category $D^b(\catA)/[m]$.
We do not explain what it means precisely. See \cite{Kel05,Zhao14} for details.

An abelian category is \emph{hereditary} 
if it is enough projective and of global dimension $1$.
Periodic derived categories of hereditary abelian categories are rather simple.
\begin{fct}[{\cite[Lemma 4.2]{Br13}, \cite[Lemma 5.1]{St18}, cf. \cite[Proposition 3.32]{S21}}]\label{fct:hered}
Let $\catA$ be a hereditary abelian category.
Then for any $m$-periodic complex $V\in D_m(\catA)$,
there exists an isomorphism $V\simeq \bigoplus H^i(V)[-i]$ in $D_m(\catA)$.
\end{fct}
In particular,
the covering functor $\pi : D^b(\catA) \to D_m(\catA)$ is essentially surjective,
and thus the $m$-periodic derived category $D_m(\catA)$ coincide with
the orbit category $D^b(\catA)/[m]$.

\begin{ex}\label{ex:path m}
Let $\bbk Q$ be a path algebra over a field $\bbk$.
By Fact \ref{fct:hered},
an indecomposable objects of $D_m(\catmod \bbk Q)$ is of the form
$M[i]$ for some $M\in \catmod \bbk Q$ and some $i \in \bbZ_m$.
If $m \ge 2$, then we have
\[
\Hom_{D_m(\catmod \bbk Q)}(M,N[i])=
\begin{cases}
\Hom_{\bbk Q}(M,N) & \text{if $i\equiv 0 \mod m$}\\
\Ext^1_{\bbk Q}(M,N) & \text{if $i\equiv 1 \mod m$}\\
0 & \text{if $i\not\equiv 0,1 \mod m$}
\end{cases}
%\Hom_{\bbk Q}(M,N),\quad
%\Hom_{D_m(\catmod \bbk Q)}(M,N)=\Ext^1_{\bbk Q}(M,N).
\]
for any $M, N \in \catmod \bbk Q$ by Fact \ref{fct:cov}.

The category $D_m(\catmod \bbk Q)$ admits 
Auslander-Reiten sequences \cite[Theorem 2.10]{Fu12}
and the covering functor $\pi : D^b(\catmod \bbk Q) \to D_m(\catmod \bbk Q)$
preserves Auslander-Reiten sequences \cite[Theorem 3.1]{Fu12}.
Thus if $Q=1 \leftarrow 2 \leftarrow 3$,
then the Auslander-Reiten quiver of $D_2(\catmod \bbk Q)$ is the following
as explained in Example \ref{ex:Iyama}.
\begin{figure}[H]
\centering
\begin{tikzpicture}
\node (1) at (0,0) {$\circ$};
\node [above right of=1] (12) {$\circ$};
\node [above right of=12] (123) {$\circ$};
\node [below right of=12] (2) {$\circ$};
\node [above right of=2] (23) {$\circ$};
\node [above right of=23] (234) {$\circ$};
\node [below right of=23] (3) {$\circ$};
\node [above right of=3] (34) {$\circ$};
\node [above right of=34] (341) {$\circ$};
\node [below right of=34] (4) {$\circ$};
\node [above right of=4] (41) {$\circ$};
\node [above right of=41] (412) {$\circ$};
 \draw[->] (1) -- (12);
 \draw[->] (12) -- (123);
 \draw[->] (12) -- (2);
 \draw[->] (123) -- (23);
 \draw[->] (2) -- (23);
 \draw[->] (23) -- (234);
 \draw[->] (23) -- (3);
 \draw[->] (234) -- (34);
 \draw[->] (3) -- (34);
 \draw[->] (34) -- (341);
 \draw[->] (34) -- (4);
 \draw[->] (341) -- (41);
 \draw[->] (4) -- (41);
 \draw[->] (41) -- (412);
 \draw[dashed] (2) -- (1);
 \draw[dashed] (3) -- (2);
 \draw[dashed] (4) -- (3);
 \draw[dashed] (23) -- (12);
 \draw[dashed] (34) -- (23);
 \draw[dashed] (41) -- (34);
 \draw[dashed] (234) -- (123);
 \draw[dashed] (341) -- (234);
 \draw[dashed] (412) -- (341);
\node [above left of=1] ('1) {};
\node [above left of=12] ('12) {};
\node [above left of=123] ('123) {};
\node [below right of=41] (41') {.};
\node [below right of=412] (412') {};
\node [above right of=412'] (4123') {};
 \draw[->] ('1) -- (1);
 \draw[->] ('12) -- (12);
 \draw[->] (41) -- (41');
 \draw[->] (412) -- (412');
 \draw[dotted] (-0.75,-0.1) -- (0.75,1.6);
 \draw[dotted] (5,-0.1) -- (6.5,1.6);
 \draw[dashed] (1) -- (-1,0);
 \draw[dashed] (12) -- ('1);
 \draw[dashed] (123) -- ('12);
 \draw[dashed] (41') -- (4);
 \draw[dashed] (412') -- (41);
 \draw[dashed] (4123') -- (412);
\end{tikzpicture}
\end{figure}
\end{ex}
%%%%%%%%%%%%%%%%%%%%%%%%%%%%%%%%%%%%%%%%%%%%%%%%%%%%%%%%%%%%%%%%%%%%%%%%%%%%%%%%%%%%%%%%%
%%%%%%%%%%%%%%%%%%%%%%%%%%%%%%%%%%%%%%%%%%%%%%%%%%%%%%%%%%%%%%%%%%%%%%%%%%%%%%%%%%%%%%%%%
\section{Grothendieck groups of periodic triangulated categories}\label{s:gro per tri}
In this section, we investigate properties of 
the Grothendieck group of a periodic triangulated category.
%%%%%%%%%%%%%%%%%%%%%%%%%%%%%%%%%%%%%%%%%%%%%%%%%%%%%%%%%%%%%%%%%%%%%%%%%%%%%%%%%%%%%%%%%
\subsection{Even periodic case}
Let $m>0$ be an even integer,
and $\catT$ be an essentially small $m$-periodic triangulated category.

\begin{lem}\label{lem:cohomo even}
A cohomological functor $F : \catT \to \catA$ induces
a homomorphism $K_0(\catT) \to K_0(\catA)$.
\end{lem}
\begin{proof}
Set $\phi(X):=\sum_{i=0}^{m-1}(-1)^i\left[F^i(X)\right] \in K_0(\catA)$.
We have to show that for any exact triangle $X \to Y \to Z \to \Sg X$ in $\catT$,
the equality $\phi(X)-\phi(Y)+\phi(Z)=0$ holds in $K_0(\catA)$.
We have two exact sequences in $\catA$:
\begin{align}\label{seq:1}
F^{m-1}(Z) \xr{f} F^m(X) \simeq F^0(X) \xr{g} F^0(Y),
\end{align}
\begin{align}\label{seq:2}
0\to \Ker g \to F^0(X) \xr{g} F^0(Y) \to \cdots 
\to F^{m-1}(Y) \to F^{m-1}(Z) \xr{f} \Ima f \to 0.
\end{align}
We get
$[\Ker g]= \phi(X)-\phi(Y)+\phi(Z) + (-1)^m[\Ima f]$
by \eqref{seq:2}.
The assumption that $m$ is even and \eqref{seq:1} imply
the equality $\phi(X)-\phi(Y)+\phi(Z)=0$.
\end{proof}

%%%%%%%%%%%%%%%%%%%%%%%%%%%%%%%%%%%%%%%%%%%%%%%%%%%%%%%%%%%%%%%%%%%%%%%%%%%%%%%%%%%%%%%%%
\subsection{Odd periodic case}
Let $m>0$ be an odd integer,
and $\catT$ be an essentially small $m$-periodic triangulated category.

\begin{lem}\label{lem:2-tor odd}
$K_0(\catT)$ is an $\bbF_2$-vector space,
that is, for any element $\alpha \in K_0(\catT)$, we have $2\alpha=0$.
\end{lem}
\begin{proof}
By the axiom of triangulated categories,
$X\to 0 \to \Sg X \to \Sg X$ is an exact triangle for any $X \in \catT$.
It implies $[\Sg X]=-[X]$ in $K_0(\catT)$.
Hence we have $[X]=[\Sg^m X]=(-1)^m[X]=-[X]$.
Thus we conclude that $2[X]=0$ for any $X\in \catT$.
\end{proof}

\begin{lem}\label{lem:cohomo odd}
A cohomological functor $\catT \to \catA$ induces
a homomorphism $K_0(\catT) \to K_0(\catA)_{\bbF_2}$,
where 
$K_0(\catA)_{\bbF_2}:=K_0(\catA)\otimes_{\bbZ}\bbF_2\simeq K_0(\catA)/2 K_0(\catA)$.
\end{lem}
\begin{proof}
Let $X \to Y \to Z \to \Sg X$ be an exact triangle in $\catT$.
Set $\phi(X):=\sum_{i=0}^{m-1}(-1)^i\left[F^i(X)\right] \in K_0(\catA)$.
Then we have, for some object $K\in \catA$,
\[
[K]= \phi(X)-\phi(Y)+\phi(Z) + (-1)^m[K]
\]
in $K_0(\catA)$ by the same calculation as in Lemma \ref{lem:cohomo even}.
Because $m$ is odd,
We get $\phi(X)-\phi(Y)+\phi(Z) = 2[K] \equiv 0 \mod 2 K_0(\catA)$.
Hence the assignment 
$[\catT] \ni[X] \mapsto \phi(X) \mod 2 K_0(\catA) \in K_0(\catA)_{\bbF_2}$
extends to a homomorphism $K_0(\catT)\to K_0(\catA)_{\bbF_2}$.
\end{proof}

%%%%%%%%%%%%%%%%%%%%%%%%%%%%%%%%%%%%%%%%%%%%%%%%%%%%%%%%%%%%%%%%%%%%%%%%%%%%%%%%%%%%%%%%%
%%%%%%%%%%%%%%%%%%%%%%%%%%%%%%%%%%%%%%%%%%%%%%%%%%%%%%%%%%%%%%%%%%%%%%%%%%%%%%%%%%%%%%%%%
\section{Proof of Theorem \ref{thm:main}}\label{s:gro per der}
Let $\catA$ be an enough projective abelian category of finite global dimension,
$m>0$ a positive integer,
and $p$ the period of $D_m(\catA)$.
There are three cases by Remark \ref{rmk:period}:
\begin{enur}
\item $m$ is even and $p=m$,
\item $m$ is odd and $p=m$, and
\item $m$ is odd and $p=2m$.
\end{enur}
We prove Theorem \ref{thm:main} separately in the three cases above. 
Note that the proof of the case (iii) also works the case (ii),
but we give separate proofs
since the proof of the case (ii) is simple and motivates the proof of (iii).

In any cases,
we have a surjective homomorphism $\psi : K_0(\catA)\to K_0\left(D_m(\catA)\right)$
induced by the natural $\delta$-functor $\catA \to D_m(\catA)$
by Proposition \ref{prp:summ}.

(i)
If $m$ is even,
the cohomology functor $H:D_m(\catA) \to \catA$ induces
a homomorphism
\[
\phi : K_0\left(D_m(\catA)\right) \to K_0(\catA),\quad
[V] \mapsto \sum_{i=1}^m (-1)^i [H^i(V)]
\]
by Lemma \ref{lem:cohomo even}.
The homomorphism $\phi$ is a retraction of $\psi$,
and hence $\psi$ is injective.
Thus $\psi$ is an isomorphism.

(ii)
If $m$ is odd and $p=m$,
then $D_m(\catA)$ is an odd periodic triangulated category.
Thus $K_0\left( D_m(\catA) \right)$ is an $\bbF_2$-vector space
by Lemma \ref{lem:2-tor odd},
and $\psi$ induces a surjective homomorphism
$\psi_{\bbF_2}: K_0(\catA)_{\bbF_2} \to K_0\left( D_m(\catA) \right)$.
The cohomology functor $H:D_m(\catA) \to \catA$ also induces a homomorphism
\[
\phi_{\bbF_2} : K_0\left(D_m(\catA)\right) \to K_0(\catA)_{\bbF_2},\quad
[V] \mapsto \sum_{i=1}^m [H^i(V)] \mod 2 K_0(\catA)
\]
by Lemma \ref{lem:cohomo odd}.
The homomorphism $\phi_{\bbF_2}$ is clearly a retraction of $\psi_{\bbF_2}$,
and hence $\psi_{\bbF_2}$ is an isomorphism.

(iii)
If $m$ is odd and $p=2m$,
then $D_m(\catA)$ is an even periodic triangulated category.
Applying Lemma \ref{lem:cohomo even} to the cohomology functor $H:D_m(\catA)\to\catA$,
we have an induced homomorphism $\phi : K_0\left(D_m(\catA)\right) \to K_0(\catA)$,
but it is a zero map.
Indeed, we have
\begin{align*}
\psi(V)
&=\sum_{i=1}^{2m}(-1)^i [H^i(V)]
=\sum_{i=1}^{m}(-1)^i [H^i(V)]+ \sum_{i=m+1}^{2m}(-1)^i[H^i(V)]\\
&=\sum_{i=1}^{m}(-1)^i [H^i(V)]-\sum_{i=1}^{m}(-1)^i [H^i(V)]
=0.
\end{align*}
Thus we cannot prove the theorem by the same way as (i).

Although $D_m(\catA)$ is even periodic,
we can prove the similar results as Lemma \ref{lem:2-tor odd} and \ref{lem:cohomo odd},
that is, $K_0\left( D_m(\catA) \right)$ is an $\bbF_2$-vector space
and the assignment $\phi_{\bbF_2}(V)=\sum_{i=1}^m [H^i(V)] \mod 2 K_0(\catA)$
defines a homomorphism 
$\phi_{\bbF_2} : K_0\left(D_m(\catA)\right) \to K_0(\catA)_{\bbF_2}$.
We first prove that $K_0\left( D_m(\catA) \right)$ is an $\bbF_2$-vector space.
For an $m$-periodic complex $V=(V^i,d^i)_{i\in\bbZ_m}$, 
$V$ and $\Sg^m V = (V^i,-d^i)_{i\in\bbZ_m}$ is not necessary isomorphic in $D_m(\catA)$
in general,
but they define the same class $[V]=[\Sg^m V]$ 
in the Grothendieck group $K_0\left( D_m(\catA) \right)$.
We prove this by induction on the number $n_V$ of $i\in \bbZ_m$ with $V^i \ne 0$.
It is clear if $n_V=0, 1$. 
Suppose $n_V\ge2$.
Then there exists $i\in \bbZ_m$ such that $V^i\ne0$.
We may assume that $i=0$.
There exists the following exact sequences in $C_m(\catA)$.
\[
\begin{tikzcd}[row sep=scriptsize,column sep=small]
0 \arrow[d] && 0 \arrow[d] & 0\arrow[d] & 0\arrow[d] \\
U \arrow[d] & : (\cdots \arrow[r] & V^{m-1} \arrow[r] \arrow[d,equal] & Z^0(V) \arrow[r,"0"] \arrow[d] & V^1 \arrow[r] \arrow[d,equal] & \cdots )\\
V \arrow[d] & : (\cdots \arrow[r] & V^{m-1} \arrow[d] \arrow[r] & V^0 \arrow[r] \arrow[d] & V^1 \arrow[r] \arrow[d] & \cdots ) \\
V^0/Z^0(V) \arrow[d] & : (\cdots \arrow[r] & 0 \arrow[d] \arrow[r] & V^0/Z^0(V) \arrow[d] \arrow[r] & 0 \arrow[r] \arrow[d] & \cdots ), \\
0&& 0 & 0 & 0
\end{tikzcd}
\]
and
\[
\begin{tikzcd}[row sep=scriptsize,column sep=small]
0 \arrow[d] && 0 \arrow[d] & 0\arrow[d] & 0\arrow[d] \\
Z^0(V) \arrow[d] & : (\cdots \arrow[r] & 0 \arrow[r] \arrow[d] & Z^0(V) \arrow[r] \arrow[d] & 0 \arrow[r] \arrow[d] & \cdots )\\
U \arrow[d] & : (\cdots \arrow[r] & V^{m-1} \arrow[r] \arrow[d,equal] & Z^0(V) \arrow[r,"0"] \arrow[d] & V^1 \arrow[r] \arrow[d,equal] & \cdots )\\
W \arrow[d] & : (\cdots \arrow[r] & V^{m-1} \arrow[d] \arrow[r] & 0 \arrow[d] \arrow[r] & V^1 \arrow[r] \arrow[d] & \cdots ) .\\
0&& 0 & 0 & 0
\end{tikzcd}
\]
Noting that $V^i=(\Sg^m V)^i$ and $Z^i(V)=Z^i(\Sg^m V)$,
we also have exact sequences $0\to \Sg^m U \to \Sg^m V \to V^0/Z^0(V)\to 0$
and $0 \to Z^0(V) \to \Sg^m U \to \Sg^m W \to 0$.
Since $n_W < n_V$, we have $[W]= [\Sg^m W]$ in $K_0\left( D_m(\catA) \right)$
by the induction hypothesis.
The canonical $\delta$-functor $C_m(\catA) \to D_m(\catA)$ carries
the exact sequences above to exact triangles in $D_m(\catA)$,
and thus we have
\[
[V]
=[W]+[Z^0(V)]+[V^0/Z^0(V)]
=[\Sg^m W] + [Z^0(V)] + [V^0/Z^0(V)]
=[\Sg^m V]=-[V].
\]
Hence $K_0\left( D_m(\catA) \right)$ is an $\bbF_2$-vector space.
The method above is Gorsky's induction technique for periodic complexes,
which appears in the proof of Fact \ref{fct:stalk}.
See \cite[Proposition 9.7]{Go13}.

Next, we prove that 
the assignment $\phi_{\bbF_2}(V)=\sum_{i=1}^m [H^i(V)] \mod 2 K_0(\catA)$
defines a homomorphism 
$\phi_{\bbF_2} : K_0\left(D_m(\catA)\right) \to K_0(\catA)_{\bbF_2}$.
Let $U\to V \to W \to \Sg U$ be an exact triangle in $D_m(\catA)$.
$m$-periodic complexes $U$ and $\Sg^m U$ are not isomorphic
but $H^m(U)=H^0(U)$ holds.
Thus we have two exact sequences in $\catA$:
\begin{align*}
H^{m-1}(W) \xr{f} H^m(U) \simeq H^0(U) \xr{g} H^0(V),
\end{align*}
\begin{align*}
0\to \Ker g \to H^0(U) \xr{g} H^0(V) \to \cdots 
\to H^{m-1}(V) \to H^{m-1}(W) \xr{f} \Ima f \to 0.
\end{align*}
A similar discussion as in Lemma \ref{lem:cohomo odd} implies
$\phi_{\bbF_2}(U) -\phi_{\bbF_2}(V) +\phi_{\bbF_2}(W) \equiv 0 \mod 2 K_0(\catA)$.

The rest of the proof is similar to (ii).
\qed
%%%%%%%%%%%%%%%%%%%%%%%%%%%%%%%%%%%%%%%%%%%%%%%%%%%%%%%%%%%%%%%%%%%%%%%%%%%%%%%%%%%%%%%%%
%%%%%%%%%%%%%%%%%%%%%%%%%%%%%%%%%%%%%%%%%%%%%%%%%%%%%%%%%%%%%%%%%%%%%%%%%%%%%%%%%%%%%%%%%
%\addcontentsline{toc}{section}{References}

%%%%%%%%%%%%%%%%%%%%%%%%%%%%%%%%%%%%%%%%%%%%%%%%%%%%%%%%%%%%%%%%%%%%%%%%%%%%%%%%%%%%%%%%%

\begin{thebibliography}{NNNN00}
\bibitem[Br13]{Br13}
 T.~Bridgeland,
 \emph{Quantum groups via Hall algebras of complexes},
 Ann.\ Math., \textbf{177} (2013), 739--759.

\bibitem[Fu12]{Fu12}
 C.\ Fu,
 \emph{On root categories of finite--dimensional algebras},
 J.\ Alg., \textbf{370} (2012), 233--265.

\bibitem[Go13]{Go13}
 M.\ Gorsky,
 \emph{Semi-derived Hall algebras and tilting invariance of Bridgeland--Hall algebras}, 
 preprint, arXiv:1303.5879. %math.QA

\bibitem[Kel05]{Kel05}
 B.~Keller,
 \emph{On triangulated orbit categories},
 Doc.\ Math., \textbf{10} (2005), 551--581. 

\bibitem[PX97]{PX97}
 L.\ Peng, J.\ Xiao,
 \emph{Root categories and simple Lie algebras},
 J.\ Alg., \textbf{198}, (1997), no.\ 1, 19--56.

\bibitem[S21]{S21}
 S.\ Saito,
 \emph{Tilting objects in periodic triangulated categories},
 preprint, arXiv:2011.14096.

\bibitem[St18]{St18}
 T.\ Stai,
 \emph{The triangulated hull of periodic complexes},
 Math.\ Res.\ Lett.\ \textbf{25} (2018), no.\ 1, 199--236.

\bibitem[Tilt07]{Tilt07}
 L.A.\ H\"ugel, D.\ Happel, H.\ Krause ed.,
 \emph{Handbook of Tilting Theory},
 London Math.\ Soc.\ Lect.\ Note Ser., \textbf{332} (2007), 
 Cambridge University Press.

\bibitem[Zhao14]{Zhao14}
 X.\ Zhao,
 \emph{A note on the equivalence of $m$-periodic derived categories},
 Sc.\ China Math., \textbf{57} (2014), no.\ 11, 2329--2334.
\end{thebibliography}
\end{document}